\documentclass[11pt, reqno]{amsart}

\usepackage{amsmath}
\usepackage{amssymb}
\usepackage{amsfonts}
\usepackage{graphicx}
\usepackage{amsthm}
\usepackage{enumerate}
\usepackage[mathscr]{eucal}
\usepackage{lscape}
\usepackage{dsfont}
\usepackage{color}
\usepackage{mathtools}

\usepackage{setspace}
\newtheorem{theor}{Theorem}

\newtheorem{defin}{Definition}
\newtheorem{rem}{Remark}
\newtheorem{prop}{Proposition}

\newtheorem{lem}{Lemma}

\newtheorem*{TA}{Theorem A} 
\newtheorem*{TB}{Theorem B}

\newcommand{\W}{\Omega}
\newcommand{\w}{\omega}
\newcommand{\deb}{\overline\partial}
\newcommand{\de}{\partial}

\newcommand{\s}{Szeg\H{o} kernel}
\newcommand{\CP}{\mathds{C}\mathrm{P}}

\newcommand{\KE}{K\"{a}hler-Einstein\ }

\newcommand{\C}{\mathbb{C}}

\newcommand{\K}{K\"{a}hler\ }

\begin{document}

\title[$\Delta$-property for complex space forms]{On the $\Delta$-property for complex space forms}
\author{Roberto Mossa}
\address{Departamento de Matemática \\
 Instituto de Matemática e Estatística \\
Universidade de São Paulo (Brazil)}
\email{robertom@ime.usp.br}

\makeatletter
\@namedef{subjclassname@2020}{%
  \textup{2020} Mathematics Subject Classification}
\makeatother

\subjclass[2020]{32M15, 32Q15, 35J05}

\keywords{\K manifolds; Hermitian symmetric spaces; \K Laplacian}

\thanks{
The author was supported by a grant from Fapesp (2018/08971-9)}

\begin{abstract} Inspired by the work of Z. Lu and G. Tian \cite{lutian}, A. Loi, F. Salis and F. Zuddas
 address in \cite{laplace} the problem of studying those \K manifolds  
satisfying the $\Delta$-property, i.e. such that  on a neighborhood of each of its points  the $k$-th power of the  \K Laplacian is a polynomial function  of the  complex Euclidean Laplacian,
for all positive integer $k$. In particular they conjectured that if  \K manifold  satisfies the $\Delta$-property then it is a complex space form. This paper is dedicated to the proof of the validity of this conjecture. 
\end{abstract}

\maketitle

\section{Introduction and statement of the main result}

Let $\Delta$ be the \K Laplacian on an $n$-dimensional \K manifold $(M, g)$ i.e., in local coordinates $z =\left\{z_j\right\}$, 
$$\Delta=\sum_{i,j=1}^n g^{i\bar j}\frac{\de^2}{\de z_j\de\bar z_i},$$
where $ g^{i\bar j}$ denotes the inverse matrix of the \K metric. We define the complex Euclidean Laplacian with respect to $z$ as 
the differential operator
$$\Delta_c^z= \sum_{i=1}^n\frac{\de^2}{\de z_i\de\bar z_i}.$$ 

A. Loi, F. Salis and F. Zuddas introduce in \cite{laplace}, the following notion of $\Delta$-property:

\begin{defin}[$\Delta$-property]\rm\label{defdprop}
For any arbitrary point  $x \in M$ there exists a coordinate system $z$ centered at $x$, such that  every  smooth function $\phi$ defined in a neighborhood of $x$  fulfills   the following equation for every positive integer $k$
 \begin{equation}\label{affine}
 \Delta^k\phi(0)=p_k(\Delta_c^z)\phi(0),
 \end{equation}
where $p_k$ is a monic polynomial of degree $k$, independent of $x$, with real coefficients.
\end{defin}

A key point in  Lu and Tian's proof of the local rigidity theorem (\cite{lutian} Theor. 1.2) supporting their conjecture about the characterization of the Fubini-Study metric $g_{FS}$ on $\CP^n$ through the vanishing of the log-term of the universal bundle, was that the $\Delta$-property is satisfied by $\left(\CP^n, g_{FS}\right)$ with respect to affine coordinates.

In \cite{laplace} A. Loi, F. Salis and F. Zuddas address the problem of studying those \K manifolds  
satisfying the $\Delta$-property. They observe that condition \eqref{affine} is satisfied for any positive integer $k$ in the center of a radial metric\footnote{Namely a \K metric admitting a \K potential which depends only on the sum $|z|^2 = |z_1|^2 + \mathellipsis+ |z_n|^2$ of the moduli of a local coordinates’ system $z$.}. It arises naturally the problem to try to classify \K manifolds satisfying the $\Delta$-property. In this direction their main results proved in \cite{laplace} are the following two theorems:
\begin{TA}{\rm(\cite[Theorem 1.3]{laplace}).}
Let $(M,g)$ be a \K  manifold which satisfies the $\Delta$-property. Then its curvature tensor is parallel.
\end{TA}
\begin{TB}{\rm (\cite[Theorem 1.4]{laplace}).}
An Hermitian symmetric space of classical type  satisfying the $\Delta$-property is a complex space form.
\end{TB}
By virtue of these two results the author conjectured in \cite{laplace} that complex space forms can be characterized as the  \K manifolds satisfying  the $\Delta$-property. The proof of this conjecture is indeed the main result of this paper. More precisely we prove the following result:
\begin{theor}\label{main}
Let $M$ be a \K manifolds satisfying the $\Delta$-property. Then $M$ is a complex space form.
\end{theor}
The proof of the theorem is based on Jordan triple system machinery. 

\medskip

The author is grateful to Prof. Andrea Loi for all the interesting discussions and the comments that helped him to improve the exposition.

\section{Proof of Theorem \ref{main}}
The first step (see Proposition \ref{hssct}) is to prove that the $\Delta$-property characterizes the complex hyperbolic space among  Hermitian symmetric spaces of noncompact type (from now on HSSNCT). Let us write $\C H^1_r$ to denote the product of $r$ complex hyperbolic spaces $\C H^1=\left\{z \in \C \mid |z|^2 <1\right\}$  equipped with  the product metric $g^r_{hyp}=g_{hyp}\oplus \mathellipsis\oplus g_{hyp}$, where the fundamental form  associated to $g_{hyp}$ is $\w_{hyp}=-\frac{i}{2}\de\deb \log (1-|z|^2)$. We call \emph{affine coordinates}, the coordinates on  $\C H^1_r$ induced by the product. The key result in our proof of Theorem \ref{main} is the following technical lemma, which extends \cite[Lemma 3.1]{laplace} valid for classical Hermitian symmetric spaces of compact type (from now on HSSCT) using Jordan triple system theory instead of Alekseevsky-Perelomov coordinates.  

\begin{lem}\label{APcoord}
Let  $M$ be an $n$-dimensional HSSNCT of rank $r$ endowed with a \KE metric $g$.
 Then there exists a normal global coordinates system $w=\left\{w_j\right\}_{j=1,\dots,n}$ such that the \K immersion's equations of  $\big(\C H^1_r,  g^r_{hyp}\big)$ into $M$ read  as 
 \begin{equation}\label{eqcoord}\begin{split} 
 \begin{cases}
 w_i=z_i &\text{for }i=1,\mathellipsis,r\\
 w_i=0 &\text{for }i=r+1,\mathellipsis,n
 \end{cases} 
 \end{split},\end{equation}
 where  $z=\left\{z_j\right\}_{j=1,\dots,r}$ are  affine  coordinates on  $\C H_r^1$.
\end{lem}
\proof
Without loss of generality we can assume that $M$ is irreducible. Throughout the proof we use Jordan triple system theory, referring the reader  to \cite{sympldual,lm1,lmz,bisy,loos,m1,mz,roos} for details and further applications.

Let $\left(V, \left\{,,\right\}\right)$ be the Hermitian positive Jordan triple system (from now on HPJTS) associated to $M$. Let $x= \lambda_1 c_1 + \dots + \lambda_s c_s$, $ \lambda_1>\dots > \lambda_s>0$  be the spectral decomposition (\cite[Definition VI.2.2]{roos}) of an element $x\in V$.   By \cite[Proposition VI.4.2]{roos}, we can realize $(M,g)$ as a bounded symmetric domain 
\begin{equation}\label{eqbsd}\begin{split} 
 \W=\left\{x\in V\mid \lambda_1<1 \right\}
\end{split}\end{equation}
equipped with the \KE form (unique up to rescaling):
$
\omega(z)=-\frac{i}{2}\de\deb \log N(z,\bar z),
$
where $N$ is the generic norm of $V$ (\cite[Section 2.2]{sympldual}).

 Consider a frame $\mathcal B =\left\{e_1,\dots,e_r\right\}\subset V$, namely a maximal set of mutually orthogonal, primitive tripotents (\cite[Definition VI.2.1]{roos}). Let $W\subset V$ be the complex vector subspace $W=\operatorname {span}_\C\!\left\{e_1,\dots,e_r\right\}$. If $x=\sum_{j=1}^{r} x_{j} e_{j}$, $y=\sum_{j=1}^{r} y_{j} e_{j},$  $z=\sum_{j=1}^{r} z_{j} e_{j}$
are elements of $W$, we have (see \cite[(6.11)]{roos})
\begin{equation*}
\{x, y, z\}=2 \sum_{j=1}^{s} x_{j} \bar y_{j} z_{j} e_{j}\in W.
\end{equation*}
Hence any frame $\left\{c_1,\dots,c_r\right\}$ of $W$ has the form 
\begin{equation}\label{eqspnorm}
c_j= e^{i\theta_j}e_{\sigma(j)}, \quad {1 \leqslant j \leqslant r}
\end{equation}
where $\sigma \in \mathfrak{S}_{r}$ is a permutation of $\{1, \ldots, r\}$ and $W$ is a Hermitian positive Jordan triple subsystem of $\left(V, \left\{,,\right\}\right)$. It is well know that there exists a one to one correspondence between sub-HPJTS $V'\subset V$ and complex totally geodesic sub-HSSNCT $\W ' \subset \W$ (see e.g. \cite[Proposition 2.1]{sympldual}), given by 
$$
V' \mapsto \W'=\W\cap V' .
$$
We want to determine the HSSNCT associated to $\left(W, \left\{,,\right\}_{\mid W}\right)$.  Let $x\in W$ and let $x= \lambda_1 c_1 + \dots + \lambda_s c_s$, $ \lambda_1>\dots > \lambda_s>0$  be its spectral decomposition (notice that, by \cite[Proposition VI.2.4]{roos}, an element $x\in W$ has the same spectral decomposition in $V$ and $W$). As recalled above, the associated HSSNCT realized as a bounded symmetric domain $\Delta^r$ of $W$ is given by
$
\Delta^r=\left\{x \in W \mid \lambda_1<1 \right\}.
$
Fixed the complex basis $\mathcal B =\left\{e_1,\dots,e_r\right\}$ defined above, we can identify $W$ with $\C^r$. With respect to this coordinates we have
$$
\Delta^r=\left\{\left(z_1,\dots,z_r\right) \mid |z_j| <1, \, j=1,\dots,r\right\}\subset \C^r
$$
 (compare this construction with \cite[Sec. 4.5.]{bisy} and \cite[Example 6]{lm1}). Denoted by $N_W$ the generic norm of $W$, we see that the associated \KE form is given by
$$
\omega_\Delta(z)=-\frac{i}{2}\de\deb \log N_W(z,\bar z)=-\frac{i}{2} \partial \bar{\partial} \log \left(\prod_{j=1}^{r}\left(1-\left|z_{j}\right|^{2}\right)\right),
$$
where we used \eqref{eqspnorm} and the fact (see e.g. \cite[Proposition VI.2.6]{roos}) that, with respect to spectral coordinates $x= \lambda_1 c_1 + \dots + \lambda_s c_s $, the generic norm $N$ of an HPJTS is given by 
$
N(x,x)=\prod_{j=1}^s\left(1-\lambda_j^2\right).
$
We conclude that $\left(\Delta, \w_\Delta\right)=\left(\C H^1_r,  g^r_{hyp}\right)$. If we complete $\mathcal B $ to a complex basis $\left\{e_1,\dots, e_r,f_1,\dots, f_{n-r}\right\}$  of $V$, we can identify $V$ with $\C^n$ obtaining coordinates $\tilde w=\left\{\tilde w_j\right\}$ for $M$ satisfying \eqref{eqcoord}. Let us choose $f_1,\dots, f_{n-r}$ in such a way that $g_{jk}(0) = g\left(\frac{\de}{\de z_j}, \frac{\de}{\de \bar z_k}\right)(0)=\delta_{jk}$, $1\leq j,k \leq n$. We introduce now (see e.g. \cite[4.17 Theorem]{ballmann}) new coordinates $w=\left\{w_j\right\}$, by solving $\tilde w_{j}={w}_{j}+\frac{1}{2} A_{k l}^{j} {w}_{k} {w}_{l}$, where $A_{k l}^{j}=-\frac{\partial g_{l \bar{j}}}{\partial z_{k}}(0)$, obtaining normal coordinates which satisfy \eqref{eqcoord} (notice that $A_{k l}^{j}=0$, for $1\leq j,k,l \leq r$). The proof is complete.
\endproof

We are now in a  position to characterize complex hyperbolic spaces among irreducible HSSCNT. This result should be compared with \cite[Theorem 3.2]{laplace}, where the authors characterize complex projective spaces among irreducible classical HSSCT via $\Delta$-property.

\begin{prop}\label{hssct}
The complex hyperbolic space is the unique HSSNCT which satisfies the $\Delta$-property.
\end{prop}
\begin{proof}
Let $M$ be an $n$-dimensional HSSNCT endowed with its \KE metric $ g$.  We denote by $\lambda$ the Einstein constant. Let $\tilde z=\left\{\tilde z_j\right\}$  be a holomorphic normal coordinate system centered in a point $x \in M$.  We have\footnote{We are going to use  the notation $\partial_i$ to denote $\frac{\partial}{\partial z_i}$ and a similar notation for higher order derivatives. We are also going to use Einstein's summation convention for repeated indices.}
 \begin{equation}\label{einstein}
 \lambda  g_{i\bar j}=\textrm{Ric}_{i\bar j}= g^{k \bar h}\left( - \partial_{ k \bar h}  g_{i\bar j}+ g^{p \bar q}\partial_k   g_{i\bar q}\partial_{\bar h}   g_{p\bar j}\right).
 \end{equation}
Hence, if we evaluate the previous equation at $0$, we get
\begin{equation}\label{sumder2}
 \sum_{h}\partial_{h\bar h}  g^{i \bar j}(0)=\lambda \delta^{i j}.  
 \end{equation}
By \eqref{sumder2}, we get
 \begin{equation}\label{laplquad}
 \Delta^2  \phi(0)=   g^{h\bar k} \partial_{k\bar h}\big(  g^{i\bar j} \partial_{j \bar i}\phi \big) \Big|_0=\Big((\Delta_c^{ \tilde z})^2+\lambda \Delta_c^{ \tilde  z}\Big)\phi(0)
 \end{equation}
that is (\ref{affine}) is satisfied also for $k=2$.

By combining \eqref{einstein}, \eqref{sumder2} and  \eqref{laplquad} above, we get that  every smooth function $\phi$ defined in a neighborhood $V$ of the origin fulfills the following
$$\Delta^3  \phi(0) =\Big((\Delta_c^{\tilde z})^3 +3\lambda(\Delta_c^{\tilde z})^2+\lambda^2\Delta_c^{\tilde z}\Big) \phi(0)+2 \sum_{ l,h=1}^{n} \partial_{l\bar h} g^{i\bar j}\partial_{j  h \bar l\bar i}\phi\ \Big|_0+$$
\begin{equation}\label{laplcube}
 +\sum_{ l,h=1}^{n} \partial_{lh} g^{i\bar j}\partial_{j \bar h \bar l\bar i}\phi\ \Big|_0+\sum_{ l,h=1}^{n} \partial_{\bar l\bar h}  g^{i\bar j}\partial_{j  h  l\bar i}\phi\ \Big|_0+\sum_{l,h=1}^{n}\partial_{l h\bar l\bar h} g^{i \bar j} \partial_{j\bar i}\phi\ \Big|_0,
 \end{equation}
 where we use that by differentiating  \eqref{einstein} and evaluating in the origin, the coefficients of the  third order derivatives of $\phi$ vanish.

Let $\left\{\tilde z_j = w_j\right\}$ be the system of normal coordinates given in Lemma \ref{APcoord} and assume (up to automorphism of $M$) that they are centered at $x$. If $\left\{z_i\right\}_{i=1,\dots,r}$ are  affine coordinates  on  $(\C H_r^1, g_{hyp}^r)$, by taking into account Lemma \ref{APcoord}, we can  compute
\begin{equation}\label{comp1}
\Delta^3  \big(|z_1|^4\big)\Big|_0={3\lambda}\big(\Delta_c^{w}\big)^2(|z_1|^4)\Big|_0+{8}\frac{\partial^2  g^{1\bar 1}}{\partial w_1\partial\overline{ w}_1}\Big|_0={12\lambda+16}.
\end{equation}
Furthermore, if $r\neq 1$, namely if $M$ is different from a complex hyperbolic space, we also compute

\begin{equation}\label{comp2}\begin{split} 
 \Delta^3   \big(|z_1z_2|^2\big)\Big|_0&={3\lambda}\big(\Delta_c^{w}\big)^2  (|z_1z_2|^2)\Big|_0\\
 &+4\left(\frac{\partial^2 g^{2\bar 2}}{\partial w_1\partial\overline{ w}_1}+ \frac{\partial^2 g^{1\bar 1}}{\partial w_2\partial\overline{ w}_2}+ \frac{\partial^2 g^{1\bar 2}}{\partial w_2\partial\overline{ w}_1}+ \frac{\partial^2 g^{2\bar 1}}{\partial w_1\partial\overline{ w}_2} \right)\Big|_0\\
 &={6\lambda}.
\end{split}\end{equation}

If $M$ has rank greater than $1$,  let us assume by contradiction that the $\Delta$-property is valid, in particular around each point of $M$ there exists a local coordinate system with respect to which \eqref{affine} is satisfied for $k=1,2,3$. Let us denote such coordinate system by $f=(f_1,\mathellipsis,f_n)$.
Since in \cite[Theorem 2.1]{laplace} is showed that every second order derivative of the holomorphic change of coordinates sending $f$ to $\tilde z$ vanish at $f=0$,
we get
$$\Delta^{3}\phi(0)=\Big((\Delta_c^{ f})^{3}+\sum_{i=1}^2 a_i(\Delta_c^{ f})^i\Big)\phi(0)=$$
$$=\Big(\sum_{i=1}^2 a_i(\Delta_c^{\tilde z})^i\Big)\phi\Big|_0+\sum_{i_1,i_2,i_{3},\alpha,\beta}\frac{\de^{3} \tilde z_\alpha}{\de { f}_{i_1} \de { f}_{i_2}\de { f}_{i_{3}}}\overline{\frac{\de^{3} \tilde z_\beta}{\de  { f}_{i_1} \de { f}_{i_2}\de { f}_{i_{3}}}}\frac{\de^2\phi}{\de \tilde z_\alpha\de\bar{\tilde z}_\beta}\Big|_0+$$
$$+(\Delta_c^{\tilde z})^{3}\phi\Big|_0+\sum_{\substack{i_1,i_2,i_{3}\\\alpha_1,\mathellipsis,\alpha_{4}}}\frac{\de^{3} \tilde z_{\alpha_{4}}}{\de { f}_{i_1} \de { f}_{i_2}\de { f}_{i_{3}}}\prod_{l=1}^{3} \overline{\frac{\de {\tilde z}_{\alpha_l}}{\de {  f}_{i_l}}}\frac{\de^{4}\phi}{\de\bar{ \tilde z}_{\alpha_1}\de\bar{ \tilde z}_{\alpha_2}\de\bar{\tilde z}_{\alpha_{3}}\de{\tilde z}_{\alpha_{4}}}\Big|_0+ $$
$$+\sum_{\substack{i_1,i_2,i_{3}\\\alpha_1,\mathellipsis,\alpha_{4}}}\overline{\frac{\de^{3} \tilde z_{\alpha_{4}}}{\de { f}_{i_1} \de { f}_{i_2}\de { f}_{i_{3}}}}\ \prod_{l=1}^{3} {\frac{\de {\tilde z}_{\alpha_l}}{\de {  f}_{i_l}}}\frac{\de^{4}\phi}{\de \tilde z_{\alpha_1}\de { \tilde z}_{\alpha_2}\de{\tilde z}_{\alpha_{3}}\de\bar{\tilde z}_{\alpha_{4}}}\Big|_0. $$
The previous formula implies   the relation
$$\Delta^3  \big(|z_1|^4\big)(0)=2\Delta^3  \big(|z_1z_2|^2\big)(0),$$ therefore we have a contradiction from the comparison with \eqref{comp1} and \eqref{comp2}. The proof is complete.
\end{proof}
\begin{rem}
\rm Notice that (see \cite[Remark 3]{laplace}) we have proved the stronger statement that the complex hyperbolic space is the unique HSSNCT such that around any point there exists a global coordinate system with respect to which (\ref{affine}) is satisfied for $k=1,2,3$.
\end{rem}

Finally, we can prove our main result.
\vspace{0.3cm}

\noindent{\it Proof of Theorem \ref{main}.} By Theorem A in the introduction,  a \K manifold $\left(M,g\right)$ satisfying the $\Delta$-property is an Hermitian symmetric space. Therefore $\left(M,g\right)$ can be decomposed as a  \K product
\begin{equation*}\label{decomposition}
(\C^n,g_0)\times (C_1,g_1)\times\mathellipsis\times( C_h,g_h)\times (N_1, \hat g_1)\times\mathellipsis\times (N_l, \hat g_l),
\end{equation*}
where $(\C^n,g_0)$ is  the flat Euclidean space, $(C_i,g_i)$ are  irreducible HSSCT and $(N_i, \hat g_i)$ are irreducible HSSNCT.

By \cite[Theorem 2.1]{laplace}, a   Hermitian symmetric space where \eqref{affine} is fulfilled for $k=1,2$, is the flat Euclidean space otherwise  it is a \K product of  Hermitian symmetric space of  either compact  or   noncompact type.  Hence, we  are going to prove our statement by characterizing the complex projective space form among  HSSCT in analogy with what we have done for hyperbolic spaces  in Proposition \ref{hssct}. 

Let $\left(C,g\right)$ be an HSSCT,  let $\left(C^*,g^*\right)$ the non compact dual and $\left(V,\left\{,,\right\}\right)$ be the associated HPJTS. Let us identify $V$ with $\C^n$ by fixing any complex basis of $V$. Then (see e.g. \cite[Section 2.4]{sympldual}) $V$ equipped with the \K form  
$$
\w_{FS} = \frac{i}{2}\de\deb \log N(z,-z)
$$
is holomorfically isometric (up to homotheties) to an open dense subset of $\left(C^*,g^*\right)$, therefore we can consider the coordinate system given in Lemma \ref{APcoord} as local coordinates for $\left(C,g\right)$. With respect this coordinates, the \K potential $\Phi=\log N(z,-z)$ for the metric $g$ satisfies 
\begin{equation}\label{ff}
\Phi (z, \bar z)=-\Phi^*(z,-\bar z)_{|C^*}
\end{equation}
where $\Phi^*  (z, \bar z)=-\log N(x,x)$ is the \K potential for $g^*$ given in \eqref{eqbsd}.

By \eqref{ff} and   \eqref{laplcube}, we get
$$\Delta_C^3\big (|z_iz_j|^2\big)(0)=-\Delta_{C^*}^3\big (|z_iz_j|^2\big)(0)$$
for every $1\leq i,j\leq \dim(C)$. Hence, if $C^*$ is not the hyperbolic space, namely if $C$ is not a complex projective space, \eqref{affine} for $k=3$ cannot be satisfied as proved in Proposition \ref{hssct}. The proof is complete.
\hfill $\Box$

\end{document}